\documentclass[11pt]{article}
\def\withcolors{0}
\def\withnotes{0}

\usepackage[backend=biber, style=alphabetic, maxbibnames=999, maxalphanames=999]{biblatex}
\def\withproofs{1}

\ifnum\withproofs=1
  \def\scalefigure{0.7}
\fi
\ifnum\withproofs=0
  \def\scalefigure{0.60}
\fi

\def\backwardsbuckets{1}

\input{packages}
\input{preamble}

\title{Identity testing under label mismatch} 

\author{Cl\'ement L. Canonne\thanks{University of Sydney. Email: \url{clement.canonne@sydney.edu.au}} \and Karl Wimmer\thanks{Duquesne University.  Email: \url{wimmerk@duq.edu}}}

\begin{document}

\maketitle

\begin{abstract}
Testing whether the observed data conforms to a purported model (probability distribution) is a basic and fundamental statistical task, and one that is by now well understood. However, the standard formulation, \emph{identity testing}, fails to capture many settings of interest; in this work, we focus on one such natural setting, \emph{identity testing under promise of permutation.} In this setting, the unknown distribution is assumed to be equal to the purported one, up to a relabeling (permutation) of the model: however, due to a systematic error in the reporting of the data, this relabeling may not be the identity. The goal is then to test identity under this assumption: equivalently, whether this systematic labeling error led to a data distribution statistically far from the reference model.
\end{abstract}

\section{Introduction}
    \label{sec:intro}
Imagine you painstakingly gathered observations, data point after data point, and managed to form an accurate estimate of the data distribution; unfortunately, you did not record the labels correctly, and due to a systematic error the data labels have been permuted in an unknown and arbitrary way. You did make your best educated guess to fix this though, and are confident the data, once carefully relabeled, \emph{should} reflect the reality.
Can you check this, without having to go through the whole process of obtaining an entirely new dataset?

In this paper, we are concerned with a variant of identity testing which captures the above scenario, where one is promised that the 
unknown distribution is equal to the reference distribution $\q$ \emph{up to a permutation of the domain}. 
Formally, the algorithm has access to i.i.d.\ samples from a probability distribution $\p$ over a finite domain $[n]\eqdef \{1,2,\dots,n\}$ such 
that $\p\circ \pi = \q$ for some (unknown) $\pi\in\mathcal{S}_n$, and, on input $0\leq \dst'<\dst \leq 1$, must output $\yes$ or $\no$ such that
\begin{itemize}
  \item if $\totalvardist{\p}{\q}\leq \dst'$, then the algorithm outputs $\yes$ with probability at least $2/3$;
  \item if $\totalvardist{\p}{\q} > \dst$, then the algorithm outputs $\no$ with probability at least $2/3$.
\end{itemize}
When $\dst'=0$, the task is termed \emph{identity testing (under promise of permutation)}; otherwise, it is \emph{tolerant identity testing}. It is worth noting that this permutation promise fundamentally changes the problem, and makes it incomparable to the standard identity testing problem. As an illustrative example, it is known that \emph{uniformity} testing, where the reference distribution $\q$ is uniform over $[n]$, is the ``hardest'' case of identity testing, with sample complexity $\bigTheta{\sqrt{n}}$ and $\bigTheta{n/\log n}$ for the testing and tolerant testing versions, respectively~\cite{Paninski08,ValiantV11,ValiantV17,Goldreich16}. However, it is easy to see that under the permutation promise, uniformity testing is a trivial problem which can be solved with \emph{zero} samples: any permutation of the uniform distribution is itself the uniform distribution.

Our results further demonstrate this stark difference, showing how the difficulty of testing and tolerant testing differ under this promise. In particular, we show an exponential gap between the sample complexities of non-tolerant and tolerant identity testing under this promise: to the best of our knowledge, this constitutes the first example of such a gap between the tolerant and non-tolerant version of a problem in distribution testing.

\subsection{Our results}
Our results show that, quite surprinsingly, the promise of equality up to permutation of the domain fundamentally changes the sample complexity landscape, and is both qualitatively and quantitatively different from what one could expect from the known bounds on identity and tolerant identity testing without this promise.

Our first set of results indeed establishes that, in contrast to the $\Theta(\sqrt{n})$ sample complexity of ``regular'' identity testing, identity testing under promise of permutation has sample complexity merely \emph{polylogarithmic} in the domain size:
\begin{theorem}[{\cref{theo:ub:testing,theo:testing:lb}, (Informal)}]
Identity testing under promise of permutation has sample complexity $\bigTheta{\log^2 n}$, where $n$ is the domain size.
\end{theorem}
Given the fact that (regular) tolerant identity testing has sample complexity nearly quadratically higher than (regular) identity testing, one could conjecture that the sample complexity tolerant testing under our promise remains polylogarithmic. Our next set of results shows that this is far from being the case: instead, allowing for some noise tolerance makes the promise of equality up to permutation essentially useless, as the sample complexity blows up \emph{exponentially}, growing from poylogarithmic to nearly linear in the domain size:
\begin{theorem}[{\cref{theo:toltesting:ub,theo:testing:lb}, (Informal)}]
Tolerant identity testing under promise of permutation has sample complexity $\bigTheta{n^{1-o(1)}}$, where $n$ is the domain size.
\end{theorem}
We also show that relaxing the tolerance allowed from additive (as in the usual tolerant testing setting) to multiplicative in the distance parameter does not really help, as the sample complexity still remains polynomial:
\begin{theorem}[{\cref{thm:tol-mult-main}, (Informal)}]
Multiplicative-factor tolerant identity testing under promise of permutation, where one needs to distinguish between $\dst$-close and $C\dst$-far, has sample complexity $\bigOmega{\sqrt{n}}$ for any constant factor $C> 1$, where $n$ is the domain size.
\end{theorem}
We emphasize once more that those results, and in particular the lower bounds, do not follow from the known results on standard identity testing, as the promise of equality up to permutation, by strenghtening the premise, drastically changes the problem. In particular, the case where the reference $\q$ is uniform, while known to be the hardest case for identity and tolerant identity testing, is actually a trivially easy case under our promise (as any distribution promised to be a permutation of the uniform distribution is, of course, the uniform distribution itself.)

\subsection{Previous work}
Distribution testing has a long history in Statistics, that one can trace back to the work of Pearson~\cite{Pearson00}. More recently, from the computer science perspective, Goldreich, Goldwasser, and Ron initiated the field of property testing~\cite{GoldreichGR98}; of which distribution testing emerged through the seminal work of Batu, Fortnow, Rubinfeld, Smith, and White~\cite{BatuFRSW00}. We refer the read to the survey~\cite{Canonne20} for a review of the area of distribution testing.

Among the problems tackled in this field, \emph{identity testing} (also known as goodness-of-fit or one-sample testing), in which the goal is to decide whether an unknown probability distribution $\p$ is equal to a purported model $\q$, has received significant attention. It is known that for identity testing with any reference distribution $\q$ over a domain of size $n$, $\Theta(\sqrt{n})$ samples are necessary and sufficient~\cite{Paninski08,ChanDVV14,AcharyaDK15,ValiantV17}; moreover, the exact asymptotic dependence on the distance parameter and the probability of error of the test~\cite{HuangM13,DiakonikolasGPP18}, as well as some good understanding of the dependence on the reference distribution $\q$ itself~\cite{ValiantV17,BlaisCG19}, are now understood. Further, we also have tight bounds for the harder problem where one seeks to allow for some noise in the data (i.e., perform \emph{tolerant} identity testing, where the algorithm has to accept distributions sufficient close to the reference $\q$): $\Theta(n/\log n)$ samples, a nearly linear dependence on the domain size, are known to be necessary and sufficient~\cite{valiantvaliant:10lb,valiantvaliant:10ub,ValiantV11,JiaoHW18}.  

However, how the identity testing problem changes under natural constraints on the input data,  or under some variations of the formulation, remains largely unexplored. Among the works concerned with such problems,~\cite{bkr:04,ddsv:13} consider identity testing under monotonicity or $k$-modality constraints; and~\cite{DiakonikolasKN15a} focuses on a broad class of shape constraints on the density. Finally,~\cite{CanonneW20} focuses on a variant of identity testing, ``identity up to binning,'' where two distributions are considered equal if some binning of the domain can make them coincide. To the best of our knowledge, the question considered in the present work, albeit arguably quite natural, has not been previously considered in the Statistics or distribution testing literature.\medskip

\noindent\textbf{Organization} We provide in~\cref{sec:testing-upper-bound} our algorithm for testing identity under promise of permutation, before complementing it in~\cref{sec:testing-lower-bound} by our matching lower bound.~\cref{sec:toleranttesting} is then concerned with the upper and lower bounds for the tolerant version of the problem; the bulk of which lies in proving the two lower bounds.

\section{Preliminaries}
    \label{sec:prelim}
Let $\mathcal{S}_n$ denote the set of permutations of $[n]\eqdef \{1,2,\dots,n\}$. We identify a probability distribution $\p$ over $[n]$ with its probability mass function (pmf), that is, a function $\p\colon[n]\to[0,1]$ such that $\sum_{i=1}^n \p(i)=1$. For a subset $S\subseteq [n]$, we then write $\p(S)=\sum_{i\in S} \p(i)$ for the probability mass assigned to $S$ by $\p$. Given two probability distributions $\p,\q$ over $[n]$, their total variation distance is
\begin{equation}
    \totalvardist{\p}{\q} = \sup_{S\subseteq [n]} (\p(S)-\q(S)) = \frac{1}{2}\normone{\p-\q}\,
\end{equation}
where $\normone{\p-\q} = \sum_{i=1}^n \abs{\p(i)-\q(i)}$ is the $\lp[1]$ distance between the two pmfs. 
In what follows, given a probability distribution $\q$ over $[n]$, we define
\begin{equation}
  \Pi_n(\q) \eqdef \setOfSuchThat{ \q\circ\pi }{ \pi\in\mathcal{S}_n },
\end{equation}
the set of distributions equal to $\q$ up to permutation of the domain.

Finally, we will rely on the so-called DKW inequality, which roughly states that $O(1/\dst^2)$ samples from any univariate distribution suffice to learn it to Kolmogorov distance $\dst$ with high probability:
this is a result due to Dvoretzky, Kiefer, and Wolfowitz from 1956~\cite{DKW:56} (with the optimal constant due to Massart, in 1990~\cite{Massart:90}).
\begin{theorem}[DKW Inequality]
  \label{theo:dkw}
  Let $\hat{\p}$ denote the empirical distribution on $m$ i.i.d.\ samples from an \emph{arbitrary} distribution $\p$ on $\R$. Then, for every $\dst > 0$,
  \[
      \probaOf{ \kolmogorov{\hat{\p}}{\p} > \dst } \leq 2e^{-2m \dst^2 }\,,
  \]
  where, for two univariate distributions $\p,\q$, $\kolmogorov{\p}{\q} = \sup_{x\in \R} \abs{ \p((-\infty,x]) - \q((-\infty,x])}$ denotes the Kolmogorov distance between $\p$ and $\q$.
\end{theorem}

\section{Testing}
In this section, we establish our matching upper and lower bounds for testing under promise of permutation,~\cref{theo:toltesting:ub,theo:testing:lb}.
\subsection{Upper bound}
    \label{sec:testing-upper-bound}
We begin by proving our $O(\log^2 n)$ upper bound for identity testing under promise of permutation.
\begin{theorem}
  \label{theo:ub:testing}
  There exists an algorithm (\cref{algo:testing}) which, for any reference distribution $\q$ over $[n]$ and any $0< \dst\leq 1$, given $\bigO{\frac{\log^2 n}{\dst^4}}$ samples from an unknown distribution $\p\in \Pi_n(\q)$, distinguishes with probability at least $2/3$ between (i)~$\p=\q$ and (ii)~$\totalvardist{\p}{\q} > \dst$.
\end{theorem}
\begin{proof}
We first partition the domain into $L\eqdef \bigO{\log(\ab/\dst)/\dst}$ \emph{buckets} $B_1,\dots,B_{L}$, where
\begin{equation}
  \label{eq:bucketing:ub}
    B_\ell \eqdef \setOfSuchThat{ i\in[\ab] }{ \frac{1}{(1+\dst/4)^{\ell}} < \q(i) \leq \frac{1}{(1+\dst/4)^{\ell-1}} }, \qquad 1\leq \ell\leq L-1
\end{equation}
and $B_L \eqdef \setOfSuchThat{ i\in[\ab] }{ \q(i) \leq \frac{1}{(1+\dst/4)^{L-1}} }$. Note that since $\q$ is known, we can exactly compute the partition $B_1,\dots,B_L$, and in particular those $L$ sets can be efficiently obtained.

\begin{algorithm}[ht]
  \begin{algorithmic}[1]
    \Require Reference distribution $\q$, distance parameter $\dst\in(0,1]$, sample access to $\p\in\Pi_n(\q)$
    \State Set $L \gets 1+ \clg{\frac{\log(4n/\dst)}{\log(1+\dst/4)}} = \bigO{\frac{\log(n/\dst)}{\dst}}$,  $\delta \gets \frac{\dst}{4(L-1)}$
    \State Compute the bucketing $B_1,\dots,B_L$, as in~\eqref{eq:bucketing:ub}
    \State Using $\bigO{1/\delta^2}$ samples from $\p$, use the empirical estimator to learn the distribution 
    \[
        \bar{\p} \eqdef (\p(B_1),\dots,\p(B_L))
    \]
    over $[L]$ to Kolmogorov distance $\frac{\delta}{3}$, with probability of error $1/10$. Let $\hat{\p}$ be the output.
    \If{ $\hat{\p}(B_L) > \frac{3\dst}{8}$ or there exists $\ell^\ast$ such that $|\hat{\p}(\{\ell^\ast,\dots,L-1\}) - \q(\bigcup_{\ell=\ell^\ast}^{L-1} B_\ell)| > \frac{\delta}{3}$ }
      \State \Return \no
    \Else
      \State \Return \yes
    \EndIf
  \end{algorithmic}
  \caption{Algorithm for identity testing under promise of permutation}\label{algo:testing}
\end{algorithm}

For our choice of $L$, $\frac{1}{(1+\dst/4)^{L-1}} \leq \frac{\dst}{4\ab}$, so the last bucket $B_L$ has small probability mass under the reference distribution: $\q(B_L)\leq \frac{\dst}{4}$. Now, distinguishing with high constant probability between $\p(B_L) \leq \frac{\dst}{4}$ and $\p(B_L) \geq \frac{\dst}{2}$ can be done with $O(1/\dst)$ samples, so we can detect a discrepancy in $B_L$ with high probability if there is one (we will argue this part formally at the end of the proof). Consequently, we hereafter assume that $\p(B_L) < \frac{\dst}{2}$.

If $\p=\q$, clearly $\p(B_L)\leq \frac{\dst}{4}$ (so the first check above passes) and $\q(B_\ell)=\p(B_\ell)$ for all $1\leq \ell\leq L-1$. However, if $\totalvardist{\p}{\q} > \dst$, then $\sum_{\ell=1}^{L-1} \sum_{i\in B_\ell}\abs{\p(i)-\q(i)} > 2\dst-\frac{3\dst}{4} = \frac{5}{4}\dst$. Moreover, letting  $\pi\in\mathcal{S}_\ab$ be the permutation such that $\p=\q\circ\pi$, consider the set $S\subseteq[n]$ of elements which $\pi$ maps to an element from the same bucket:
\[
    S \eqdef \setOfSuchThat{ i \in [n]\setminus B_L }{ \exists \ell \in [L-1],\, i\in B_\ell,\pi(i) \in B_\ell }\,.
\]
For each such element $i$, by definition of the bucketing, $\abs{\p(i)-\q(i)} =  \abs{\q(i)-\q(\pi(i))} \leq \frac{\dst}{4}\q(i)$. It follows that the elements from $S$ amount for a total $\lp[1]$ distance of at most $\frac{\dst}{4}$, and therefore a constant fraction of the distance between $\p$ and $\q$ comes from the set $T\eqdef[n]\setminus (B_L\cup S)$ of elements that $\pi$ ``moves to a different bucket:''
\[
    \frac{5}{4}\dst 
        < \sum_{i\in S} \abs{\p(i)-\q(i)} + \sum_{i\in T} \abs{\p(i)-\q(i)}
        \leq \frac{\dst}{4}\q(S) + \sum_{i\in T} \abs{\p(i)-\q(i)}
        \leq \frac{\dst}{4} + \sum_{i\in T} \abs{\p(i)-\q(i)}
\]
that is, $\sum_{i\in T} \abs{\p(i)-\q(i)} > \dst$.

Partition the set $T$ by setting $T_\ell \eqdef T\cap B_\ell$, for $\ell\in[L-1]$. Rewriting the above inequality, we obtained that
\begin{equation}
  \label{eq:testing:proof}
    \sum_{i\in T} \abs{\p(i)-\q(i)} = \sum_{\ell=1}^{L-1}\sum_{i\in T_\ell} \abs{\p(i)-\q(i)} > \dst\,.
\end{equation}
We will use this to prove the following result.
\begin{claim}
  \label{claim:test:distance}
  Suppose that $\totalvardist{\p}{\q} > \dst$. Then there exists some $\ell^\ast\in[L-1]$ such that $\abs{\p(\bigcup_{\ell=\ell^\ast}^{L-1} B_\ell) - \q(\bigcup_{\ell=\ell^\ast}^{L-1} B_\ell)} > \delta$, where $\delta = \frac{\dst}{4(L-1)}$.
\end{claim}
\begin{proof}
We note that since $\p(B_\ell) = \p(S_\ell)+\p(T_\ell)$ and that $\p(S_\ell)=\q(S_\ell)$ (by definition of $S_\ell\subseteq S$) for every $\ell$, it suffices to prove the statement for $T_\ell$, that is, that there exists $\ell^\ast$ such that
\[
  \abs{ \p(\bigcup_{\ell=\ell^\ast}^{L-1} T_\ell) - \q(\bigcup_{\ell=\ell^\ast}^{L-1} T_\ell) } > \delta\,.
\]
The key property we will use is that, for every $\ell < \ell'$, we have $\q(i) \geq \q(j)$ for every $i\in B_\ell, j\in B_{\ell'}$. This property, which follows from the definition of bucketings, guarantees that if $\pi$ maps an element $i\in T_{\ell'}$ to element $\pi(i)\in T_{\ell}$, then $\p(i) \geq \q(i)$. 

Let $U,V \subseteq [L-1]$ be the buckets whose probability mass under $\p$ is greater than or equal to (resp., less than or equal to) the probability mass under $\q$, i.e.,
\[
    U \eqdef \setOfSuchThat{ \ell \in[L-1] } { \p(T_\ell) \geq \q(T_\ell) }, \quad V \eqdef \setOfSuchThat{ \ell \in[L-1] } { \p(T_\ell) \leq \q(T_\ell) }
\]
This lets us rewrite~\eqref{eq:testing:proof} as
\[
    \dst < \sum_{\ell\in U}\sum_{i\in T_\ell} \abs{\p(i)-\q(i)} + \sum_{\ell\in V}\sum_{i\in T_\ell} \abs{\p(i)-\q(i)}
\]
and so at least one of the two terms in the RHS must exceed $\frac{\dst}{2}$. Without loss of generality, suppose $\sum_{\ell\in U}\sum_{i\in T_\ell} \abs{\p(i)-\q(i)} > \frac{\dst}{2}$. This implies there exists $\ell^\ast\in U$ such that $\sum_{i\in T_{\ell^\ast}} \abs{\p(i)-\q(i)} > \frac{\dst}{2(L-1)}$; we will focus on this $\ell^\ast$.\medskip

Partition $T_{\ell^\ast}$ further into $T^+_{\ell^\ast}$ and $T^-_{\ell^\ast}$, where $T^+_{\ell^\ast}$ (resp. $T^-_{\ell^\ast}$) is the set of elements $i\in T_{\ell^\ast}$ such that $\pi(i)$ belongs to a bucket $B_\ell$ with $\ell < \ell^\ast$ (resp., $\ell > \ell^\ast$). Note that, for any $i\in T^+_{\ell^\ast}$, we then have $\p(i) = \q(\pi(i)) \geq \q(i)$, and conversely for $i\in T^-_{\ell^\ast}$: so that we can rewrite the above as
\[
    \frac{\dst}{2(L-1)} 
    <  (\p(T^+_{\ell^\ast})-\q(T^+_{\ell^\ast})) + (\q(T^-_{\ell^\ast})-\p(T^-_{\ell^\ast}))
\]
Now, since $\p(T_{\ell^\ast}) \geq \q(T_{\ell^\ast})$ (as $\ell^\ast\in U$), we have $\p(T^+_{\ell^\ast})-\q(T^+_{\ell^\ast}) \geq \q(T^-_{\ell^\ast})-\p(T^-_{\ell^\ast})$ and therefore
$\p(T^+_{\ell^\ast})-\q(T^+_{\ell^\ast}) > \frac{\dst}{4(L-1)}$. This implies the claim: indeed, we then have
\[
    \p(\cup_{\ell=\ell^\ast}^{L-1} T_\ell) > \q(\cup_{\ell=\ell^\ast}^{L-1} T_\ell) + \frac{\dst}{4(L-1)}
\]
since $T_{\ell^\ast}$ ``receives'' a difference of at least $\frac{\dst}{4(L-1)}$ probability mass from lower-index buckets, and besides this the total probability mass of the suffix of buckets $\cup_{\ell=\ell^\ast}^{L-1} T_\ell$ cannot decrease by any internal swap of elements.\qedhere
\end{proof} 
With the above claim in hand, we can conclude the analysis. Indeed, as the sets $B_1,\dots,B_{L}$ are known, one can estimate the induced probability distribution 
$
    \bar{\p} \eqdef (\p(B_1),\dots,\p(B_L))
$
to Kolmogorov distance $\frac{\delta}{3}$ (with probability at least $9/10$) using $O(1/\delta^2) = O(L^2/\dst^2) =O(\log^2(n/\dst)/\dst^4)$ samples (this follows from~\cref{theo:dkw}). Let $\hat{\p}$ be the resulting distribution over $[L]$. Whenever this step is successful (i.e., with probability at least $9/10$, the following holds.
\begin{itemize}
  \item If $\p=\q$, then $|\hat{\p}(L)-\q(B_L)| \leq \frac{\delta}{3}$, so $\hat{\p}(B_L)\leq \frac{\dst}{4}+\frac{\delta}{3} \leq \frac{3}{8}\dst$; and 
$|\hat{\p}(\{\ell^\ast,\dots,L-1\}) - \q(\bigcup_{\ell=\ell^\ast}^{L-1} B_\ell)| \leq \frac{\delta}{3}$ for all $\ell$. Thus, the test accepts.
  \item If $\totalvardist{\p}{\q} > \dst$, then either
    \begin{itemize}
        \item $\p(B_L) > \frac{\dst}{2}$, in which case $\hat{\p}(L) \geq \frac{\dst}{2} - \frac{\delta}{3} > \frac{3}{8}\dst$ and the test rejects; or
        \item $\p(B_L) \leq \frac{\dst}{2}$, in which case by~\cref{claim:test:distance} there exists some $\ell^\ast\in[L-1]$ such that 
        $\abs{\p(\bigcup_{\ell=\ell^\ast}^{L-1} B_\ell) - \q(\bigcup_{\ell=\ell^\ast}^{L-1} B_\ell) } > \delta$. 
        Then, 
        \[
            \abs{ \hat{\p}(\{\ell^\ast,\dots, L\}) - \q(\bigcup_{\ell=\ell^\ast}^{L-1} B_\ell) }\geq \abs{\p(\bigcup_{\ell=\ell^\ast}^{L-1} B_\ell) - \q(\bigcup_{\ell=\ell^\ast}^{L-1} B_\ell) }- \frac{\delta}{3} > \frac{2}{3}\delta
        \]
        and the test rejects.
    \end{itemize}
\end{itemize}
This concludes the proof of correctness of the algorithm. The claimed sample complexity readily follows from our choice of $\delta = \bigTheta{L/\dst}$ and the $O(1/\delta^2)$ sample complexity of learning an arbitrary real-valued distribution to Kolmogorov distance $\delta$.
\end{proof}

\begin{remark}[On the tolerance of the tester]
  \label{rk:small:tolerance}
  We note that the above analysis establishes a slightly stronger statement; namely, that the testing algorithm allows for some small tolerance, accepting distributions that are $O(\dst/\log \ab)$-close to $\q$, and rejecting those that are $\dst$-far. As we will see later, this $\Omega(\log \ab)$ factor in the amount of tolerance is essentially optimal, as by~\cref{thm:tol-mult-main} reducing it to $o(\log \ab)$ would require sample complexity $\ab^{1/2-o(1)}$.
\end{remark}

\subsection{Lower bound}
    \label{sec:testing-lower-bound}
In this section, we show that the $O(\log^2 n)$ upper bound from the previous section is tight, by proving a matching lower bound on the sample complexity of identity testing under promise of permutation.
\begin{theorem}
  \label{theo:testing:lb}
  Any algorithm which, given a reference distribution $\q$ over $[n]$, $0< \dst \leq 1$ such that $\dst = \tildeOmega{1/\ab^{1/4}}$, and sample access to an unknown distribution $\p\in \Pi_{\ab}(\q)$, distinguishes with probability at least $2/3$ between (i)~$\p=\q$ and (ii)~$\totalvardist{\p}{\q} > \dst$, must have sample complexity $\bigOmega{\frac{\log^2 \ab}{\dst^2}}$.
\end{theorem}
\begin{proof}
We first describe a construction with constant distance $\dst=1/9$, leading to an $\bigOmega{\log^2 n}$ lower bound; before explaining how to obtain the claimed $\bigOmega{\frac{1}{\dst^2}\log^2 n}$ lower bound from it.
Our lower bound will rely on a reference distribution $\q$ piecewise-constant on $L=\bigTheta{\log\ab}$ buckets, where bucket $\ell$ has a number of elements proportional to $2^\ell$. The first and last buckets (that is, the smallest and largest) will each have total probability mass $1/3$ under $\q$, and be uniform. The remaining ``middle'' $L-2$ buckets all have $1/(3(L-2))$ total probability mass, and are uniform as well.  We then build a family of perturbations $\{\p_\pi = \q\circ \pi\}_\pi\subseteq \Pi_n(\q)$, such that under each perturbation $\p_\pi$ the middle buckets keep the exact same total probability mass $1/(3(L-2))$, by ``cascading'' mass from one bucket to the next. Details follow.

Set $L\eqdef \bigTheta{\log \ab}$ to be the largest integer such that $L 2^L \leq \sqrt{\ab}$, and assume for convenience that $\clg{\sqrt{\ab}}$ is a multiple of $3$. The $\ell$th bucket $B_\ell$, for $0\leq \ell\leq L-2$, has size
\[
    |B_\ell| = \clg{\sqrt{\ab}}\cdot 2^\ell
\]
and $|B_{L-1}| = 2(L-2)|B_{L-2}|$, so that 
\[
    \frac{\ab}{8}\leq \sum_{\ell=0}^{L-1} |B_\ell| = \clg{\sqrt{\ab}} \cdot 2^{L-1} \mleft( L-1 \mright)  \leq \ab\,.
\] (We hereafter focus on the first part of the domain, and will ignore the last $\ab-\sum_{\ell=0}^{L-1} |B_\ell|$ elements.) Note that each bucket contains at least $\sqrt{\ab}$ elements by construction, and has a size which is a multiple of $3$. The reference distribution $\q$ is then uniform inside each bucket, where
\begin{itemize}
  \item $\q(B_0)=\q(B_{L-1})=\frac{1}{3}$, and
  \item $\q(B_\ell)=\frac{1}{3(L-2)}$ for all $0 < \ell < L-1$.
\end{itemize}
In particular, our choice of $|B_{L-1}|$ ensures that each element of the last bucket, under $\q$, will have probability mass
\[
  \frac{1}{3|B_{L-1}|} = \frac{1}{2} \cdot \frac{1}{3(L-2)|B_{L-2}|}
\]
that is, half the probability mass of elements of the $(L-2)$th bucket. 

Each perturbation will then have the same distribution over buckets:
\begin{itemize}
  \item each of the $L-2$ middle buckets $B_\ell$ is (independently) partitioned uniformly at random into $3$ sets $S_{\ell,1}, S_{\ell,2}, S_{\ell,3}$ of equal size. The permutation then swaps $S_{\ell,2}\cup S_{\ell,3}$ and $S_{\ell+1,1}$, for $1\leq \ell\leq L-3$ (note that indeed $\abs{S_{\ell,2}\cup S_{\ell,3}} = \abs{S_{\ell+1,1}}$, but $\q(S_{\ell,2}\cup S_{\ell,3}) = 2\q(S_{\ell+1,1})=\frac{2}{9(L-2)}$).
  \item a uniformly random subset $S_0\subseteq B_0$ of size $\frac{|B_1|}{3(2L-5)} = O( |S_{1,1}|/L)$ is selected, and the permutation swaps it with a uniformly random subset $T_1\subseteq S_{1,1}$ of equal size. By choice of the size, we had $\q(S_0) = \frac{2}{9(2L-5)}$ and $\q(T_1) = \frac{1}{9(2L-5)(L-2)}$, so that $\q(S_0)-\q(T_1) = \frac{1}{9(L-2)}$.
  \item similarly, the subset $S_{L-2,2}\cup S_{L-2,3}$ of size $\frac{2}{3}|B_{L-2}| = \frac{|B_{L-1}|}{3(L-2)}$ is swapped with a uniformly random subset $T_{L-1}\subseteq B_{L-1}$ of equal size. By choice of the size, we had $\q(S_{L-2,2}\cup S_{L-2,3}) = \frac{2}{9(L-2)}$ and $\q(T_{L-1}) = \frac{1}{9(L-2)}$, so that again $\q(S_{L-2,2}\cup S_{L-2,3})-\q(T_{L-1}) = \frac{1}{9(L-2)}$.
  	
\end{itemize}
As a result, we get that for each such perturbation $\p=\q\circ\pi$, $\totalvardist{\p}{\q} \geq \frac{1}{9}$. The construction is illustrated in~\cref{fig:lb:testing}.

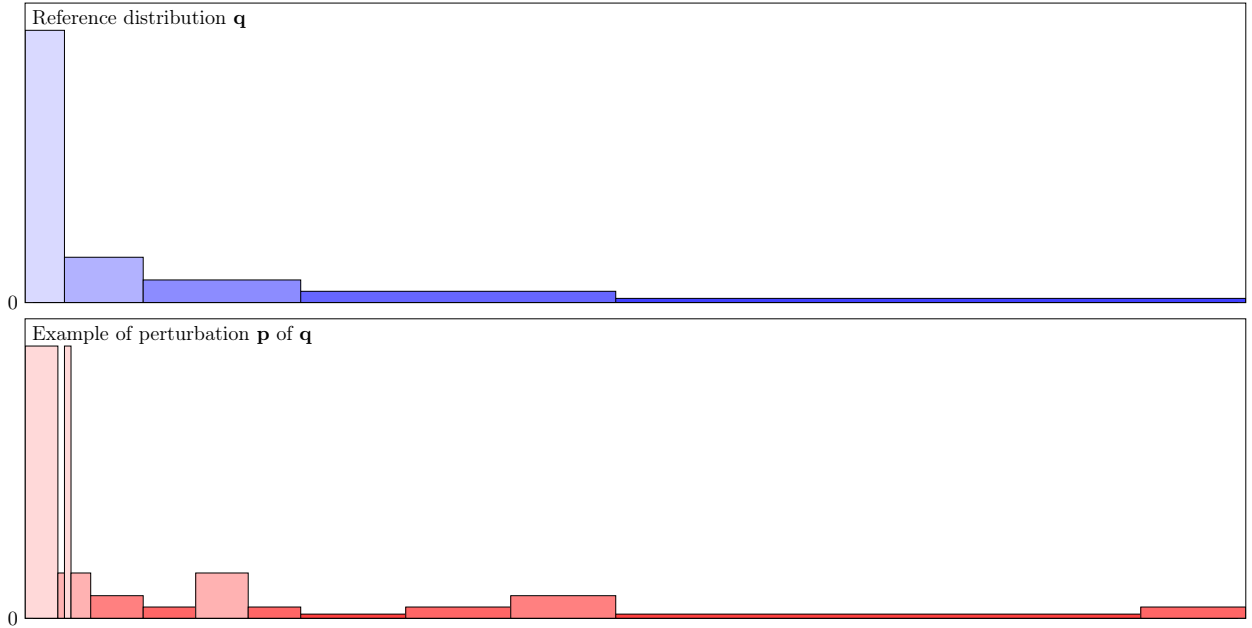
\begin{figure}[ht!]\centering
\begin{tikzpicture}[scale=\scalefigure]
\begin{axis}[
    xmin=0, xmax=186, ymin=0,
    ytick={0}, xtick=\empty,
    width=1.5\textwidth,
    height=\axisdefaultheight,
    area style,
    ]
\node[anchor=north west] at (rel axis cs:0,1) {Reference distribution $\q$};
\addplot+[ybar interval,mark=no,color=blue!15,draw=black] plot coordinates { (3*2*(2^(1-1)-1), 1/3*1/2.0^1) (3*2*(2^1-1), 0) };
\foreach \r in {2,...,4}{
  \pgfmathtruncatemacro{\i}{15*\r}
  \edef\temp{\noexpand%
    \addplot+[ybar interval,mark=no,color=blue!\i,draw=black] plot coordinates { (3*2*(2^(\r-1)-1), 1/9*1/2.0^\r) (3*2*(2^\r-1), 0) };
  }\temp
}
\addplot+[ybar interval,mark=no,color=blue!75,draw=black] plot coordinates { (3*2*(2^(5-1)-1), 1/3*1/2.0^5/4) (3*2*(2^5-1), 0) };
\end{axis}
\end{tikzpicture}

\begin{tikzpicture}[scale=\scalefigure]
\begin{axis}[
    xmin=0, xmax=186, ymin=0,
    ytick={0}, xtick=\empty,
    width=1.5\textwidth,
    height=\axisdefaultheight,
    area style,
    ]
\node[anchor=north west] at (rel axis cs:0,1) {Example of perturbation $\p$ of $\q$};
\addplot+[ybar interval,mark=no,color=red!15,draw=black] plot coordinates { (3*2*(2^(1-1)-1), 1/3*1/2.0^1) (3*2*(2^1-1)-1, 0) };
\addplot+[ybar interval,mark=no,color=red!30,draw=black] plot coordinates { (3*2*(2^1-1)-1, 1/9*1/2.0^2) (3*2*(2^1-1), 0) };

\addplot+[ybar interval,mark=no,color=red!15,draw=black] plot coordinates { (3*2*(2^(2-1)-1), 1/3*1/2.0^1) (3*2*(2^(2-1)-1)+1, 0) };
\addplot+[ybar interval,mark=no,color=red!30,draw=black] plot coordinates { (3*2*(2^(2-1)-1)+1, 1/9*1/2.0^2) (3*2*(2^(2-1)-1)+3*2^2*1/3, 0) };
\addplot+[ybar interval,mark=no,color=red!50,draw=black] plot coordinates { (3*2*(2^(2-1)-1)+2^2, 1/9*1/2.0^3) (3*2*(2^(2-1)-1)+3*2^2*3/3, 0) };

\addplot+[ybar interval,mark=no,color=red!60,draw=black] plot coordinates { (3*2*(2^(3-1)-1), 1/9*1/2.0^4) (3*2*(2^(3-1)-1)+3*2^3*1/3, 0) };
\addplot+[ybar interval,mark=no,color=red!30,draw=black] plot coordinates { (3*2*(2^(3-1)-1)+3*2^3*1/3, 1/9*1/2.0^2) ((3*2*(2^(3-1)-1)+3*2^3*2/3, 0) };
\addplot+[ybar interval,mark=no,color=red!60,draw=black] plot coordinates { ((3*2*(2^(3-1)-1)+3*2^3*2/3, 1/9*1/2.0^4) ((3*2*(2^(3-1)-1)+3*2^3*3/3, 0) };

\addplot+[ybar interval,mark=no,color=red!75,draw=black] plot coordinates { (3*2*(2^(4-1)-1), 1/3*1/2.0^5/4) (3*2*(2^(4-1)-1)+3*2^4*1/3, 0) };
\addplot+[ybar interval,mark=no,color=red!60,draw=black] plot coordinates { (3*2*(2^(4-1)-1)+3*2^4*1/3, 1/9*1/2.0^4) (3*2*(2^(4-1)-1)+3*2^4*2/3, 0) };
\addplot+[ybar interval,mark=no,color=red!50,draw=black] plot coordinates { ((3*2*(2^(4-1)-1)+3*2^4*2/3, 1/9*1/2.0^3) ((3*2*(2^(4-1)-1)+3*2^4*3/3, 0) };

\addplot+[ybar interval,mark=no,color=red!75,draw=black] plot coordinates { (3*2*(2^(5-1)-1), 1/3*1/2.0^5/4) (3*2*(2^5-1)-3*2^4*1/3, 0) };
\addplot+[ybar interval,mark=no,color=red!60,draw=black] plot coordinates { ((3*2*(2^5-1)-3*2^4*1/3, 1/9*1/2.0^4) (3*2*(2^5-1), 0) };
\end{axis}
\end{tikzpicture}
\caption{\label{fig:lb:testing}Reference distribution $\q$ and example of perturbation $\p$, for $L=5$. Note that the total probability mass of each bucket of $\q$ is preserved under $\p$, except for the first and last one whose mass decreases and increases by $\Theta(1/L)$, respectively.} 
\end{figure}

By a birthday paradox-type argument, no element will be sampled twice unless the number of samples is at least $\Omega(1/\sqrt{\sum_{i=1}^n \p(i)^2})=\Omega(1/\sqrt{n \max_{i \in [n]} \p(i)})=\Omega(\ab^{1/4})$, which is far beyond the polylogarithmic regime we are working in.
By construction, under each $\p$, all $L-2$ middle buckets have exactly the same probability mass $\frac{1}{3(L-2)}$, and elements inside are perturbed randomly, either having probability (compared to $\q$) multiplied by $2$ with probability $1/3$ or divided by $2$ with probability $1/3$. Because of the uniformly random choice of the 3-way partition inside each bucket and the fact that each of all those inner partitions are chosen independently across buckets, the information from those $L-2$ buckets does not provide any advantage in distinguishing them from $\p$ unless the same element is hit twice.\footnote{That is, conditioned on seeing each element of those $L-2$ buckets at most once, the conditional distribution over those $L-2$ buckets under (i)~$\q$ and (ii)~the uniform mixture of all perturbations $\p$ are indistinguishable.}

This addresses the case of the middle $L-2$ buckets. Turning to the remaining two, the probability mass of both end buckets, under any perturbation $\p$, deviates from what it is under $\q$ by an additive $\delta\eqdef \frac{1}{9(L-2)}$. Since those buckets each have total probability mass $1/3$ under $\p$ and $1/3\pm\delta$ under each $\q$ and we do not see any collisions with high probability, detecting this requires $\Omega(1/\delta^2)=\Omega(\log^2 \ab)$ samples, giving the lower bound for constant $\dst=1/9$.\smallskip
\ignore{By construction, under each $\p$, all $L-2$ middle buckets have exactly the same probability mass $\frac{1}{3(L-2)}$, and elements inside are perturbed randomly, either having probability (compared to $\q$) multiplied by $2$ with probability $1/3$ or divided by $2$ with probability $1/3$. Because of the uniformly random choice of the 3-way partition inside each bucket and the fact that each of all those inner partitions are chosen independently across buckets, the information from those $L-2$ buckets does not provide any advantage in distinguishing them from $\p$ unless the same element is hit twice.\footnote{That is, conditioned on seeing each element of those $L-2$ buckets at most once, the conditional distribution over those $L-2$ buckets under (i)~$\q$ and (ii)~the uniform mixture of all perturbations $\p$ are indistinguishable.} However, by a birthday paradox-type argument, no element within those $L-2$ buckets will be sampled twice unless the number of samples is at least $\Omega(1/\sqrt{\sum_{\ell=1}^{L-2} \sum_{i\in B_\ell} \p(i)^2})=\Omega(1/\sqrt{\sum_{\ell=1}^{L-2} 1/|B_\ell|})=\Omega(\ab^{1/4})$.
	
	This addresses the case of the middle $L-2$ buckets. Turning to the remaining two, the probability mass of both end buckets, under any perturbation $\p$, deviates from what it is under $\q$ by an additive $\delta\eqdef \frac{1}{9(L-2)}$. Since those buckets each have total probability mass $1/3$ under $\p$ and $1/3\pm\delta$ under each $\q$, detecting this requires $\Omega(1/\delta^2)=\Omega(\log^2 \ab)$ samples, giving the lower bound for constant $\dst=1/9$.\smallskip}

To obtain the inverse quadratic dependence on the distance parameter, one can then simply repeat the above argument for any $0<\dst<1/9$ by replacing our reference distribution $\q$ and all the perturbations $\p_\pi=\q\circ\pi$ by the mixtures
\[
    \q_\dst \eqdef (1-9\dst)\uniform + 9\dst \q, \qquad \p_{\dst,\pi} \eqdef (1-9\dst)\uniform + 9\dst \p_\pi = \q_\dst\circ\pi
\]
the last equality crucially using the fact that the uniform distribution $\uniform$ (over the domain) is invariant by permutation. Note that every such $\p_{\dst,\pi}$ then does belong to $\Pi_n(\q_\dst)$, and is at total variation distance exactly $\dst$ from $\q_\dst$. Moreover, we can repeat the previous argument \emph{mutatis mutandis}: (i) the middle buckets provide no information whatsoever unless an element is seen twice, which requires $\bigOmega{\ab^{1/4}/\dst}$ samples (the extra $1/\dst$ due to our mixture with weight $9\dst$); while the two outer buckets have a discrepancy only $\delta\eqdef \frac{\dst}{L-2}$, which to be detected requires at least $\Omega(1/\delta^2)=\Omega((\log^2 \ab)/\dst^2)$ samples overall. The minimum of these two quantities gives the claimed lower bound, as long as $\ab^{1/4}/\dst =\Omega((\log^2 \ab)/\dst^2)$ , that is, $\dst = \bigOmega{(\log^2 n)/\ab^{1/4}}$.
\end{proof}

\section{Tolerant testing}
    \label{sec:toleranttesting}

We now turn to the task of \emph{tolerant} testing. As mentioned in the introduction, tolerant testing is well known to be harder than standard (non-tolerant) testing, with a nearly quadratic gap for the standard identity testing problem ($\sqrt{n}$ vs. $\frac{n}{\log n}$ sample complexity). Surprisingly, we are able to show that under the promise of permutation, the task does not suffer a merely polynomial blowup -- the sample complexity of tolerant identity testing becomes \emph{exponentially} harder than that of standard testing, jumping from $\log^2 n$ to $n^{1-o(1)}$.

The first component, an $\bigO{{n}/{\log n}}$ upper bound for tolerant testing under promise of permutation (\cref{theo:toltesting:ub}), is straightforward, and simply follows from the corresponding upper bound absent this promise. A much more challenging task is in establishing the lower bound. We actually provide two lower bounds: the first, an $\Omega(n^{1-o(1)})$ lower bound (\cref{theo:toltesting:lb:1}), applies for the usual setting of tolerant testing with an additive gap $\delta$ between $\dst'$ and $\dst$). The second (\cref{thm:tol-mult-main}) is an $\bigOmega{\sqrt{n/2^{O(C)}}}$ sample complexity lower bound for any $C$-factor approximation of the distance, that is to distinguish between $\dst$-close and $C\dst$-far.

\subsection{Upper bound}

The claimed upper bound readily follows from the analogous upper bound on tolerant testing \emph{without} the promise of permutation, due to Valiant and Valiant~\cite[Theorem~4]{ValiantV11} (see, also,~\cite{JiaoHW18}). Indeed, any such estimator can be used for our problem, ignoring the additional promise of identity up to permutation.
\begin{theorem}
  \label{theo:toltesting:ub}
  There exists an algorithm which, for any reference distribution $\q$ over $[n]$ and any $0\leq \eps,\delta\leq 1$ such that $\delta = \bigOmega{1/\sqrt{\log n}}$, and given $\bigO{\frac{n}{\delta^2\log n}}$ samples from an unknown distribution $\p\in \Pi_n(\q)$, distinguishes with probability at least $2/3$ between (i)~$\totalvardist{\p}{\q} \leq \dst$ and (ii)~$\totalvardist{\p}{\q} > \dst+\delta$.
\end{theorem}
We note that the requirement $\delta = \bigOmega{1/\sqrt{\log n}}$ has been relaxed in~\cite{JiaoHW18}.

\subsection{Lower bound}
In this section, we prove the theorem below, our lower bound on the sample complexity of tolerant testing under promise of permutation. Before doing so, we emphasize that the known $\bigOmega{\frac{n}{\delta^2\log n}}$ sample complexity lower bound for tolerant testing \emph{absent} this promise does not apply to our setting, as the promise of permutation makes the testing problem easier. In particular, the hard instances used to prove the aforementioned $\bigOmega{\frac{n}{\delta^2\log n}}$ lower bound do not satisfy this promise.\footnote{One can also note that the lower bound for ``standard'' tolerant testing is obtained by choosing the reference distribution to be uniform over $[n]$. Under promise of permutation, this particular instance of the problem is trivial, as any permutation of the uniform distribution is still the uniform distribution.}
\begin{theorem}
  \label{theo:toltesting:lb:1}
  Any algorithm which, given a reference distribution $\q$ over $[n]$, $0< \dst, \delta \leq 1$, and sample access to an unknown distribution $\p\in \Pi_n(\q)$, distinguishes with probability at least $2/3$ between (i)~$\totalvardist{\p}{\q} \leq \dst$ and (ii)~$\totalvardist{\p}{\q} > \dst+\delta$, must have sample complexity $\bigOmega{\delta^2 n^{1-O(1/\log(1/\delta))}}$.
\end{theorem}
\begin{proof}
In what follows, we assume that $\delta = \bigOmega{1/\sqrt{n}}$, as otherwise there is nothing to prove. Let $k\geq 1$ be an integer to be chosen during the course of the analysis (we will set $k=\Theta(1/\delta)$), and write $n=2mk^2$ for some integer $m\geq 1$ (this can be done without loss of generality, as our assumption on $\delta$ ensures that $n\geq 2mk^2$). For $1\leq \ell\leq 2k$, we define the integer interval $I_{k,\ell}\eqdef[k]+(\ell-1)k$, so that $[2k^2] = \bigcup_{\ell=1}^{2k}I_{k,\ell}$.

Given two distributions $\p,\q$ over $[k]$, we define families of distributions $\closefam_{\p,\q}$ and $\farfam_{\p,\q}$ over $[n]$ as follows: first, we consider the distributions $\closedist,\fardist$, each over $[2k^2]$, obtained by ``repeating and alternating'' $\p$ and $\q$ as follows:
\begin{itemize}
	\item For $1\leq \ell\leq k$ and $j \in \bucket{\bucketsz}{\ell}$, $\closedist(j) = \frac{1}{2k}\p(j)$.
	\item For $1\leq \ell\leq k$ and $j \in \bucket{\bucketsz}{k+\ell}$, $\closedist(j) = \frac{1}{2k}\q(j)$.
\end{itemize}

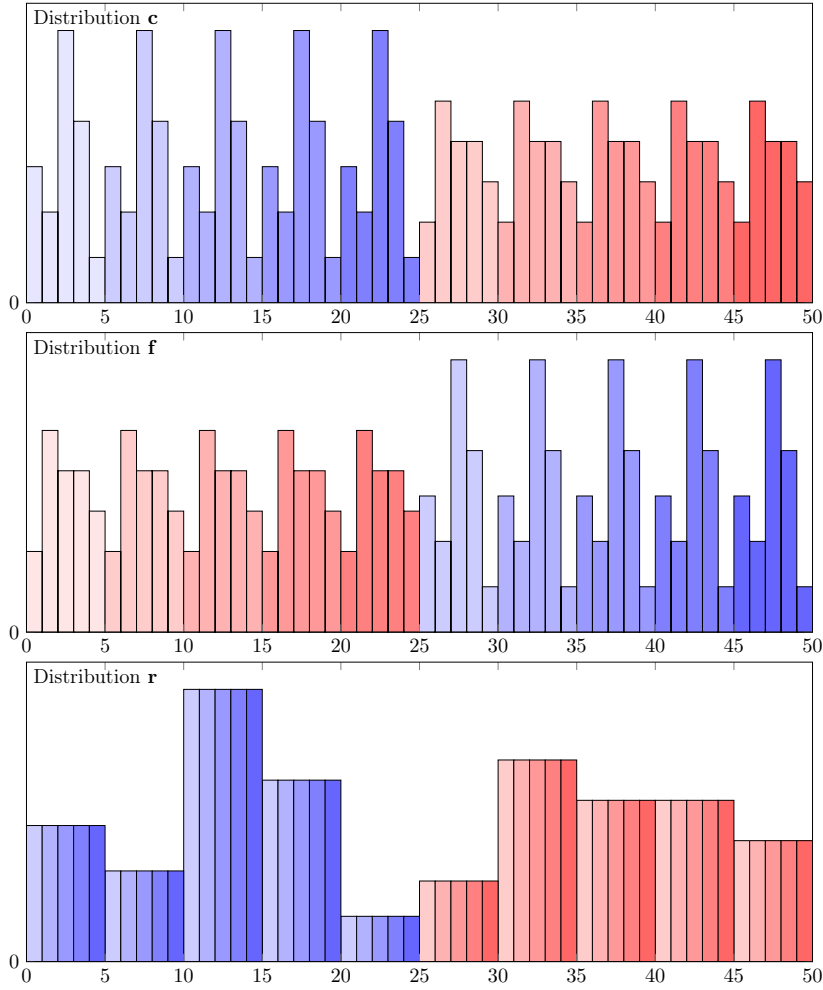
\begin{figure}[ht]\centering
\begin{tikzpicture}[scale=\scalefigure]
\begin{axis}[
    xmin=0, xmax=50, ymin=0, 
    ytick={0},
    width=\textwidth,
    height=\axisdefaultheight,
    area style,
    ]
\node[anchor=north west] at (rel axis cs:0,1) {Distribution $\closedist$};
\foreach \r in {0,...,4}{
  \pgfmathtruncatemacro{\i}{10*\r+10}
  \edef\temp{\noexpand%
    \addplot+[ybar interval,mark=no,color=blue!\i,draw=black] plot coordinates { (5*\r, 3/16) (5*\r+1, 2/16) (5*\r+2, 6/16) (5*\r+3, 4/16) (5*\r+4, 1/16) (5*\r+5, 0) };
  }\temp
}
\foreach \r in {5,...,9}{
  \pgfmathtruncatemacro{\i}{10*(\r-5+1)+10}
  \edef\temp{\noexpand%
  \addplot+[ybar interval,mark=no,color=red!\i,draw=black] plot coordinates { (5*\r, 4/36) (5*\r+1, 10/36) (5*\r+2, 8/36) (5*\r+3, 8/36) (5*\r+4, 6/36) (5*\r+5, 0) };
  }\temp
}
\end{axis}
\end{tikzpicture}
\\
\begin{tikzpicture}[scale=\scalefigure]
\begin{axis}[
    xmin=0, xmax=50, ymin=0, 
    ytick={0},
    width=\textwidth,
    height=\axisdefaultheight,
    area style,
    ]
\node[anchor=north west] at (rel axis cs:0,1) {Distribution $\fardist$};
\foreach \r in {5,...,9}{
  \pgfmathtruncatemacro{\i}{10*(\r-5+1)+10}
  \edef\temp{\noexpand%
    \addplot+[ybar interval,mark=no,color=blue!\i,draw=black] plot coordinates { (5*\r, 3/16) (5*\r+1, 2/16) (5*\r+2, 6/16) (5*\r+3, 4/16) (5*\r+4, 1/16) (5*\r+5, 0) };
  }\temp
}
\foreach \r in {0,...,4}{
  \pgfmathtruncatemacro{\i}{10*\r+10}
  \edef\temp{\noexpand%
  \addplot+[ybar interval,mark=no,color=red!\i,draw=black] plot coordinates { (5*\r, 4/36) (5*\r+1, 10/36) (5*\r+2, 8/36) (5*\r+3, 8/36) (5*\r+4, 6/36) (5*\r+5, 0) };
  }\temp
}
\end{axis}
\end{tikzpicture}
\\
\begin{tikzpicture}[scale=\scalefigure]
\begin{axis}[
    xmin=0, xmax=50, ymin=0, 
    ytick={0},
    width=\textwidth,
    height=\axisdefaultheight,
    area style,
    ]
\node[anchor=north west] at (rel axis cs:0,1) {Distribution $\refdist$};
\foreach \r in {1,...,5}{
  \pgfmathtruncatemacro{\i}{10*\r+10}
  \edef\temp{\noexpand%
    \addplot+[ybar interval,mark=no,color=blue!\i,draw=black] plot coordinates { (\r-1, 3/16) (\r, 0) };
  }\temp
  \edef\temp{\noexpand%
    \addplot+[ybar interval,mark=no,color=blue!\i,draw=black] plot coordinates { (5+\r-1, 2/16) (5+\r, 0) };
  }\temp
  \edef\temp{\noexpand%
    \addplot+[ybar interval,mark=no,color=blue!\i,draw=black] plot coordinates { (10+\r-1, 6/16) (10+\r, 0) };
  }\temp
  \edef\temp{\noexpand%
    \addplot+[ybar interval,mark=no,color=blue!\i,draw=black] plot coordinates { (15+\r-1, 4/16) (15+\r, 0) };
  }\temp
  \edef\temp{\noexpand%
    \addplot+[ybar interval,mark=no,color=blue!\i,draw=black] plot coordinates { (20+\r-1, 1/16) (20+\r, 0) };
  }\temp
}

\foreach \r in {1,...,5}{
  \pgfmathtruncatemacro{\i}{10*\r+10}
  \edef\temp{\noexpand%
    \addplot+[ybar interval,mark=no,color=red!\i,draw=black] plot coordinates { (25+\r-1, 4/36) (25+\r, 0) };
  }\temp
  \edef\temp{\noexpand%
    \addplot+[ybar interval,mark=no,color=red!\i,draw=black] plot coordinates { (25+5+\r-1, 10/36) (25+5+\r, 0) };
  }\temp
  \edef\temp{\noexpand%
    \addplot+[ybar interval,mark=no,color=red!\i,draw=black] plot coordinates { (25+10+\r-1, 8/36) (25+10+\r, 0) };
  }\temp
  \edef\temp{\noexpand%
    \addplot+[ybar interval,mark=no,color=red!\i,draw=black] plot coordinates { (25+15+\r-1, 8/36) (25+15+\r, 0) };
  }\temp
  \edef\temp{\noexpand%
    \addplot+[ybar interval,mark=no,color=red!\i,draw=black] plot coordinates { (25+20+\r-1, 6/36) (25+20+\r, 0) };
  }\temp
}
\end{axis}
\end{tikzpicture}
\caption{An example of $\closedist$ (top), $\fardist$ (middle), and $\refdist$ (bottom) over $[2k^2]$, for $k=5$; here, we took $\p=\frac{1}{16}(3,2,6,4,1)$ and $\q=\frac{1}{18}(2,5,4,4,3)$.}
\end{figure}

\noindent We obtain $\fardist$ over $[2k^2]$ in a similar fashion, but swapping $\bucket{\bucketsz}{\ell}$ and $\bucket{\bucketsz}{k+\ell}$:
\begin{itemize}
	\item For $1\leq \ell\leq k$ and $j \in \bucket{\bucketsz}{\ell}$, $\fardist(j) = \frac{1}{2k}\q(j)$.
	\item For $1\leq \ell\leq k$ and $j \in \bucket{\bucketsz}{k+\ell}$, $\fardist(j) = \frac{1}{2k}\p(j)$.
\end{itemize}

\noindent Further, we define our ``reference'' distribution $\refdist$ over $[2k^2]$ as
\begin{itemize}
	\item For $1 \leq \ell \leq \bucketsz$ and $j \in \bucket{\bucketsz}{\ell}$, $\refdist(j) = \frac{1}{2k}\p(\ell)$.
	\item For $\bucketsz+1 \leq \ell \leq 2\bucketsz$ and $j \in \bucket{\bucketsz}{\ell}$, $\refdist(j) = \frac{1}{2k}\q(\ell)$.	
\end{itemize}
We also define the reference distribution $\refdist^\ast_{\p,\q}$ over $[n]=[2k^2m]$ by concatenating $m$ copies of $\refdist$ and normalizing the result; that is,
\[
    \refdist^\ast_{\p,\q} \eqdef \frac{1}{2m}(\refdist \sqcup \refdist \sqcup \dots \sqcup \refdist),
\]
where $\sqcup$ denotes the vector concatenation.
Note that both $\closedist$ and $\fardist$ are permutations of $\refdist$, and that $\normone{\refdist}=\normone{\closedist}=\normone{\fardist}=1$. 
Next, we bound the gap between $\totalvardist{\fardist}{\refdist}$ and $\totalvardist{\closedist}{\refdist}$, relating it to the distance between $\p$ and $\q$.
\begin{claim}
  \label{claim:distance:gap:c:f:r}
$\totalvardist{\fardist}{\refdist} \geq \totalvardist{\closedist}{\refdist} + \frac{1}{k}\totalvardist{\p}{\q}$
\end{claim}
\begin{proof}
	We will analyze the contributions to $\totalvardist{\closedist}{\refdist}$ and $\totalvardist{\closedist}{\fardist}$ on $\bucket{k}{\ell}$ and $\bucket{k}{k+\ell}$ for $1 \leq \ell \leq k$. Without loss of generality, we can assume that $\p,\q$ are non-decreasing. Then, from our definition of $\closedist$, $\refdist$, and $\fardist$, we have
\begin{align*}
\totalvardist{\fardist}{\refdist} 
&= \frac{1}{4k} \sum_{i=1}^k \sum_{j=1}^k (|\p(i) - \q(j)| + |\q(i) - \p(j)|) 
= \frac{1}{2k} \mleft( \sum_{i=1}^k \sum_{j=1}^k |\p(i) - \q(j)|  \mright)\\
&= \frac{1}{2k} \mleft( \sum_{i=1}^k |\p(i)-\q(i)| + \sum_{i=1}^k\sum_{j=1}^{i-1} (|\p(i)-\q(j)| + |\p(j)-\q(i)|) \mright)\,\\
\totalvardist{\closedist}{\refdist} 
&= \frac{1}{4k} \sum_{i=1}^k \sum_{j=1}^k (|\p(i) - \p(j)| + |\q(i) - \q(j)|)
= \frac{1}{2k} \sum_{i=1}^k\sum_{j=1}^{i-1} ((\p(i)-\p(j))+(\q(i)-\q(j)))
\end{align*}
where for the last equality we used the assumption that $\p,\q$ were non-decreasing to write
\[
\sum_{i=1}^k \sum_{j=1}^k |\p(i) - \p(j)| 
= \sum_{i=1}^k \sum_{j=1}^{i-1} (\p(i)-\p(j)) + \sum_{i=1}^k \sum_{j=i+1}^{k} (\p(j)-\p(i)) = 2\sum_{i=1}^k \sum_{j=1}^{i-1} (\p(i)-\p(j))\,,
\]

The conclusion then follows from recalling that $\totalvardist{\p}{\q} = \frac{1}{2}\sum_{i=1}^k |\p(i)-\q(i)|$, and observing that
$(\p(i)-\p(j))+(\q(i)-\q(j)) = (\p(i)-\q(j))+(\q(i)-\p(j)) \leq |\p(i)-\q(j)| + |\p(j)-\q(i)|$.\qedhere

\end{proof}
To define $\closefam_{\p,\q}$ and $\farfam_{\p,\q}$, we will need one further piece of notation. We denote by $\mathcal{B}_{k}\subseteq\mathcal{S}_{2k^2}$ the set of all permutations of $[2k^2]$ ``respecting the buckets,'' that is,
\[
    \mathcal{B}_{k} \eqdef \setOfSuchThat{ \pi \in \mathcal{S}_{2k^2}}{ \pi(I_{k,\ell})=I_{k,\ell} \forall \ell \in [2k] }
\]
We then let
\[
    \closefam_{\p,\q} = \setOfSuchThat{  \frac{1}{2mk}(\closedist\circ\pi_1 \sqcup \closedist\circ\pi_2 \sqcup \dots \sqcup \closedist\circ\pi_m) }{ \pi_1,\dots, \pi_m \in \mathcal{B}_{k} }
\]
and 
\[
    \farfam_{\p,\q} = \setOfSuchThat{  \frac{1}{2mk}(\fardist\circ\pi_1 \sqcup \fardist\circ\pi_2 \sqcup \dots \sqcup \fardist\circ\pi_m) }{ \pi_1,\dots, \pi_m \in \mathcal{B}_{k} }
\]
where as before $\sqcup$ denotes the vector concatenation; that is, we stitch together $m$ blocks, each consisting on a permuted version of either $\closedist$ or $\fardist$. Note that since $n=m\cdot 2k^2$ and each $\closedist$ (resp. $\fardist$) is a $(2k^2)$-dimensional vector, $\closefam_{\p,\q}$ and $\farfam_{\p,\q}$ are indeed families of probability distributions over $[n]$, and 
$\closefam_{\p,\q},\farfam_{\p,\q} \subseteq \Pi_{n}(\refdist^\ast_{\p,\q})$.\medskip

The construction above allows us to convert any two distributions $\p,\q$ with sufficiently many matching moments to families of distributions (whose elements are all permutations of a single reference one) hard to distinguish:
\begin{claim}
  \label{claim:lower:bound:from:moment:matching}
  There exists some absolute constant $c>0$ such that, if $\p,\q$ have matching first $r$-way moments, it is impossible to distinguish a uniformly random element of $\closefam_{\p,\q}$ from a uniformly random element of $\farfam_{\p,\q}$ given fewer than $c m^{1-\frac{1}{r+1}}$ samples.
\end{claim}
\begin{proof}
By assumption on $\p,\q$ and out construction of $\closedist,\fardist$ from them, for every of the $m$ contiguous blocks of $2k^2$ elements, the $r$-way moments of the corresponding conditional distributions exactly match. Given that a uniformly element drawn of $\p'$ from $\closefam_{\p,\q}$ and $\q'$ from  $\farfam_{\p,\q}$ corresponds to independent permutations inside each block, any block in which fewer than $r+1$ samples falls brings exactly zero information about whether it comes from $\p'$ or $\q'$ (specifically, one could simulate the distribution of those $s < r+1$  samples without getting any sample from the real distribution). Since each of these $m$ blocks has total probability $1/m$ under both $\p'$ and $\q'$, by a generalized birthday paradox (see, e.g.,~\cite{SuzukiTKT06}), with probability at least $9/10$ no block will receive more than $r$ samples unless the total number of samples is at least $c m^{1-\frac{1}{r+1}}$, for some absolute constant $c>0$.
\end{proof}

It remains to specify \emph{which} pair of distributions with ``sufficiently many matching moments'' we will use. While we could argue directly about the existence of such a pair of distributions with desirable properties, it is simpler to leverage a construction due to Valiant and Valiant~\cite{ValiantV11}, which exhibits the desired properties.
\begin{claim}
  \label{claim:moment:matching}
  There exists some $\dst_0>0$ such that the following holds. For every sufficiently large $r$, there exists a pair of distributions (without loss of generality, non-decreasing) $\p_{\rm VV},\q_{\rm VV}$ over $k=O(r2^r)$ elements with matching first $r$-way moments, but $\totalvardist{\p_{\rm VV}}{\q_{\rm VV}} \geq \dst_0$. %
\end{claim}
\begin{proof}
This follows from the lower bound construction of~\cite{ValiantV11}.
\end{proof}
We will rely on this pair of distributions $\p_{\rm VV},\q_{\rm VV}$, and hereafter write $\closefam,\farfam,$ and $\refdist^\ast$ for $\closefam_{\p_{\rm VV},\q_{\rm VV}},\farfam_{\p_{\rm VV},\q_{\rm VV}},$ and $\refdist^\ast_{\p_{\rm VV},\q_{\rm VV}}$, respectively.
\begin{claim}
For every $\p'\in\closefam$ and $\q'\in\farfam$, we have $\totalvardist{\q'}{\refdist^\ast} > \totalvardist{\p'}{\refdist^\ast} + \frac{\dst_0}{k}$.
\end{claim}
\begin{proof}
Due to the definition of $\closefam$, $\farfam$, and $\refdist^\ast$ as $m$-fold concatenations, and since $\refdist$ is invariant by permutations from $\mathcal{B}_{k}$, it is sufficient to prove the claim for $\p_{\rm VV}$, $\q_{\rm VV}$, and $\refdist$ (over $[2k^2]$). The claimed bound then immediately follows from~\cref{claim:distance:gap:c:f:r}.
\end{proof}
To finish the argument, it only remains to combine the various claims. We choose $k\geq \frac{\dst_0}{\delta}$ and $m = n/(2k^2) \geq 1$ (since $\delta = \Omega(1/\sqrt{n})$. By~\cref{claim:moment:matching}, we can then set $r\eqdef \bigOmega{\log k}$ and obtain, from~\cref{claim:lower:bound:from:moment:matching}, a sample complexity lower bound of
\[
    \bigOmega{ m^{1-\frac{1}{r+1}} } = \bigOmega{\delta^2 n^{1-\bigO{\frac{1}{\log(1/\delta)}}}}
\] 
as desired.
\end{proof}

\ifnum\withproofs=1
The theorem immediately implies the following two corollaries.
\begin{corollary}
  For every $c>0$, there exists some $\delta > 0$ such that the following holds. 
  Any algorithm which, given a reference distribution $\q$ over $[n]$, $\dst\in(0,1)$, and sample access to an unknown distribution $\p\in \Pi_n(\q)$, distinguishes with probability at least $2/3$ between (i)~$\totalvardist{\p}{\q} \leq \dst$ and (ii)~$\totalvardist{\p}{\q} > \dst+\delta$, must have sample complexity $\bigOmega{ n^{1-c} }$.
\end{corollary}

\begin{corollary}
Any algorithm which, given a reference distribution $\q$ over $[n]$, $\dst\in(0,1)$, and sample access to an unknown distribution $\p\in \Pi_n(\q)$, distinguishes with probability at least $2/3$ between (i)~$\totalvardist{\p}{\q} \leq \dst$ and (ii)~$\totalvardist{\p}{\q} > \dst+1/2^{\sqrt{\log n}}$, must have sample complexity $\frac{n}{2^{O(\sqrt{\log n})}}$.
\end{corollary}
\fi

\ifnum\backwardsbuckets=0
\paragraph{Tolerant testing $C$-approximation}
\newcommand{\twocmone}{m}
\renewcommand{\bucket}{B}
\newcommand{\apx}{C}
\newcommand{\refsum}{s}
We now turn to our second tolerant testing lower bound, which applies to algorithms providing a $C$-factor approximation of the distance to the reference distribution.
\begin{theorem}
\label{thm:tol-mult-main}
	Any algorithm which, given a reference distribution $\q$ over $[n]$, $0< \dst \leq 1$, $\apx \geq 2$, and sample access to an unknown distribution $\p\in \Pi_n(\q)$, distinguishes with probability at least $2/3$ between (i)~$\totalvardist{\p}{\q} \leq \dst$ and (ii)~$\totalvardist{\p}{\q} > \apx \dst$, must have sample complexity $\bigOmega{\sqrt{\frac{n}{8^\apx}}/\dst}$.
\end{theorem}

\begin{proof}
	We will prove the theorem via a sequence of lemmas. 
\ifnum\withproofs=0
In the interest of space, their proofs are deferred to the full version of the paper. 
\fi
	 We will assume that $\apx \geq 2$ is an integer, and we define $\twocmone = 2^\apx-1$.  Our proof will proceed similarly to the proof of Theorem~\ref{theo:toltesting:lb:1}. We will begin by working over $[\twocmone(2^{\apx+1}+2^{\apx-1}-3)]$.  Throughout this section, we partition $[\twocmone(2^{\apx+1}+2^{\apx-1}-3)]$ into $\apx+1$ buckets, which we will denote $\bucket_0,\bucket_1,\ldots,\bucket_{C}$, such that each $\bucket_i$ is a set of consecutive integers, $|\bucket_0| = \twocmone2^{\apx-1}$, $|\bucket_\apx| = \twocmone$, and $|\bucket_i| = \twocmone2^{\apx+1-i}$ for $1 \leq i \leq \apx-1$.  We define $\bucket_i = \{b_i,b_i + 1,\ldots,b'_i\}$.

\noindent We define a distribution $\refdist$ in the following way:
\begin{itemize}
	\item For each $j \in \bucket_0$, $\refdist(j) = \frac1{\refsum}$.
	\item For each $1 \leq i \leq \apx-1$ and $j \in \bucket_i$, $\refdist(j) = \frac{2^i}{\refsum}$.
	\item For each $j \in \bucket_\apx$, $\refdist(j) = \frac{2^{\apx}}{\refsum}$.
\end{itemize}

We define two distributions $\p$ and $\q$ such that $\p$ and $\q$ are hard to distinguish with few samples, such that $\totalvardist{\refdist}{\p}$ and $\totalvardist{\refdist}{\q}$ are far apart. 
We define $\q$ in the following way:
\begin{itemize}
	\item For each $j \in \bucket_0$, $\q(j) = \frac2{\refsum}$.
	\item For each $1 \leq i \leq \apx-1$,
	\begin{itemize}
		\item For $j$ in $\{b_i,\ldots,b_i+\twocmone2^{\apx-i}-1\}$, $\q(j) = \frac{2^{i-1}}{\refsum}$.
		\item For $j$ in $\{b_i+\twocmone2^{\apx-i},\ldots,b_i+\twocmone(2^{\apx-i} + 2^{\apx-1-i}) - 1\}$, $\q(j) = \frac{2^i}{\refsum}$.
		\item For $j$ in $\{b_i+\twocmone(2^{\apx-i} + 2^{\apx-1-i}),\ldots,b'_i\}$, $\q(j) = \frac{2^{i+1}}{\refsum}$.
	\end{itemize} and $j \in \bucket_i$, $\q(j) = \frac{2^i}{\refsum}$.
	\item For each $j \in \bucket_\apx$, $\q(j) = \frac{2^{\apx-1}}{\refsum}$.
\end{itemize}

\noindent We define $\p$ as follows:
\begin{itemize}
	\item For each $j \in \bucket_0$,
	\begin{itemize}
		\item If $j \in \{b_0,\ldots,b_0+\twocmone2^{\apx-1}-1\}$, then $\p(j) = \frac1{\refsum}$.
		\item If $j \in \{b_0+\twocmone2^{\apx-1},\ldots,b'_0\}$, then $\p(j) = \frac{2^\apx}{\refsum}$.
	\end{itemize}
	\item For each $1 \leq \apx-1$ and $j \in \bucket_i$, $\p(j) = \refdist(j) = \frac{2^i}{\refsum}$.
	\item For each $j \in \bucket_\apx$,
	\begin{itemize}
		\item If $j \in \{b_\apx,\ldots,b_\apx+2^{\apx-1}-1\}$, then $\p(j) = \frac1{\refsum}$.
		\item If $j \in \{b_\apx+2^{\apx-1},\ldots,b'_\apx\}$, then $\p(j) = \frac{2^\apx}{\refsum}$.
	\end{itemize}
\end{itemize}

\begin{lemma}
\label{lem:tol-mult-equal-buckets}
For $0 \leq i \leq \apx$, $\sum_{j \in \bucket_i} \p(j) = \sum_{j \in \bucket_i} \q(j)$.
\end{lemma}
\ifnum\withproofs=1
\begin{proof}
	The proof is simply direct calculation.	 Observe that in bucket $0$, 
	\[
	\refsum \sum_{j \in \bucket_0} \q(j) = \twocmone2^{\apx-1} \cdot 2 (\twocmone-1)2^{\apx-1} + (\twocmone+1)2^{\apx-1} = 1 \cdot (\twocmone-1)2^{\apx-1} + 2^{\apx} \cdot 2^{\apx-1} = \refsum \sum_{j \in \bucket_0} \p(j).
	\]
	
	In bucket $\apx$, we have\	\[
	\refsum \sum_{j \in \bucket_\apx} \q(j) = \twocmone \cdot 2^{\apx-1} = (\twocmone - 1)2^{\apx-1} + 2^{\apx-1} = (2^{\apx} - 2)2^{\apx-1} + 2^{\apx-1} = 2^{\apx} \cdot (2^{\apx-1} - 1) + 1 \cdot 2^{\apx-1} = \refsum \sum_{j \in \bucket_\apx} \p(j).
	\]
	
	For $1 \leq i \leq \apx-1$, we have 
	\begin{align*}
	\refsum \sum_{j \in \bucket_\apx} \p(j) &= 2^i \cdot \twocmone 2^{\apx+1-i} \\
	&= \twocmone 2^{\apx-1-i} 2^{i+2} \\
	&= \twocmone 2^{\apx-1-i} (2(2^{i-1}) + 2^i + 2^{i+1}) \\
	&= \twocmone 2^{\apx-1-i} \cdot 2^{i-1} + \twocmone2^{\apx-1-i} \cdot 2^i + \twocmone2^{\apx-1-i} \cdot 2^{i+1} \\
	&= \refsum \sum_{j \in \bucket_\apx} \q(j).
	\end{align*}
	
	The claim follows by dividing the equalities by $\refsum$.
\end{proof}
\fi
\begin{lemma} 
	\label{lem:tv-distance-far}
	$\totalvardist{\refdist}{\q} = \frac{C}{4C-1}$
\end{lemma}
\ifnum\withproofs=1
\begin{proof}
	By direct calculation,
	\begin{align*}
	2\refsum\totalvardist{\refdist}{\q} &= \refsum \sum_{j=1}^{\refsum} |\refdist(j) - \q(j)| \\
	&= \twocmone 2^{\apx-1}(2-1) + \twocmone(2^\apx - 2^{\apx-1}) + \frac12 \sum_{i=1}^{\apx-1} \left( \twocmone2^{\apx-i}(2^i - 2^{i-1}) + \twocmone2^{\apx-1-i}(2^{i+1}-2^i) \right) \\
	&= \twocmone 2^{\apx} + \sum_{i=1}^{\apx-1} (2^{i-1}\twocmone2^{\apx-i} + \twocmone2^{\apx-1-i}2^i) \\
	&= \twocmone 2^{\apx} + \sum_{i=1}^{\apx-1} (\twocmone2^{\apx-1} + \twocmone2^{\apx-1}) \\
	&= \apx \twocmone 2^{\apx}.
	\end{align*}
Dividing both sides by $2\refsum$ yields the lemma.
\end{proof}
\fi
\begin{lemma}
\label{lem:tol-mult-buckets-almost-uniform}
	For every $0 \leq i \leq \apx$, $\frac{1}{4\apx-1} \leq \p(\bucket_i) \leq \frac{4}{4\apx-1}$ (and similarly for $\q(\bucket_i))$.
\end{lemma}
\ifnum\withproofs=1
\begin{proof}
	By Lemma~\ref{lem:tol-mult-equal-buckets}, it suffices to check either $\p(\bucket_i)$ or $\q(\bucket_i)$ for each $0 \leq i \leq \apx$.  For bucket $0$, we get
	
	\[
	\q(\bucket_0) = \frac{2}{\refsum} \cdot \twocmone{2^\apx-1} = \frac{2}{4C-1}.
	\]
	
	For bucket $C$, we get
	
	\[
	\q(\bucket_{\apx}) = \frac{2^{\apx-1}}{\refsum} \cdot \twocmone = \frac{1}{4\apx-1}
	\]
	
	For $1 \leq i \leq C-1$, we get
	
	\[
	\p(\bucket_i) = \frac{2^i}{\refsum} \cdot \twocmone(2^{\apx+1-i}) = \frac{4}{4\apx-1}.
	\]
	
\end{proof}
\fi
\begin{lemma} 
\label{lem:tv-distance-close}
	$\totalvardist{\refdist}{\p} = \frac{1}{4C-1}$
\end{lemma}
\ifnum\withproofs=1
\begin{proof}
	By direct calculation,
	\[
	2\refsum\totalvardist{\refdist}{\p} = 2^{\apx-1} \cdot (2^\apx - 1) + 2^{\apx-1} \cdot (2^\apx - 1) = 2^{\apx} (2^{\apx-1}) = \twocmone 2^{\apx}.
	\]
	Dividing both sides by $2\refsum$ yields the lemma.
\end{proof}
\fi
\newcommand{\numblk}{t}

We assume that $n$ is a multiple of $\refsum$, and define $\numblk := \frac{n}{\refsum}$.
To define $\closefam$ and $\farfam$ over [n], we will need one further piece of notation. We denote by $\mathcal{B}'_{\refsum}\subseteq\mathcal{S}_{\refsum}$ the set of all permutations of $[\refsum]$ ``respecting the buckets,'' that is, for every $0 \leq i \leq \apx$,
\[
\mathcal{B}'_{\refsum} = \{ \pi \in \mathcal{S}_{\refsum} : \pi(\bucket_i) = \bucket_i \forall i \in \{0,1,\ldots,C\} \}
\]
We then let
$
\refdist^* := \frac{1}{\numblk} (\refdist \sqcup \refdist \sqcup \cdots \sqcup \refdist)
$
as well as
\begin{align*}
\closefam &= \setOfSuchThat{  \frac{1}{\numblk}(\closedist\circ\pi_1 \sqcup \closedist\circ\pi_2 \sqcup \dots \sqcup \closedist\circ\pi_\numblk) }{ \pi_1,\dots, \pi_\numblk \in \mathcal{B'}_{\refsum} }
\\
\farfam &= \setOfSuchThat{  \frac{1}{\numblk}(\fardist\circ\pi_1 \sqcup \fardist\circ\pi_2 \sqcup \dots \sqcup \fardist\circ\pi_\numblk) }{ \pi_1,\dots, \pi_\numblk \in \mathcal{B'}_{\refsum} }
\end{align*}
where as before $\sqcup$ denotes vector concatenation.  Since $\totalvardist{\refdist}{\closedist \circ \pi} = \totalvardist{\refdist}{\closedist}$ and $\totalvardist{\refdist}{\fardist \circ \pi} = \totalvardist{\refdist}{\fardist}$ for all $\pi \in \mathcal{B}'_{\refsum}$, we have that $\totalvardist{\refdist^*}{\p} = \frac{1}{4\apx-1}$ for every distribution $\p \in \closefam$, and $\totalvardist{\refdist^*}{\q} = \frac{\apx}{4\apx-1}$ for every distribution $\q \in \farfam$.  Further, repeating the same partitioning of each interval of $\refsum$ elements of $[n]$ into buckets $\bucket_0,\bucket_1,\ldots,\bucket_C$, we have $\numblk (\apx+1)$ buckets, such that distinguishing a distribution in $\closefam$ from a distribution in $\farfam$ requires seeing at $2$ samples in at least one of these buckets.
Since the probability mass on each bucket is in the interval $[\frac{1}{\numblk(\apx+1)},\frac{4}{\numblk(\apx+1)}]$, at least $\Omega(\sqrt{\numblk(\apx+1)}) = \Omega(\sqrt{n(\apx+1)/\refsum})$ queries to distinguish in $\closefam$ from a distribution in $\farfam$, completing the proof of Theorem~\ref{thm:tol-mult-main}.
\end{proof}

\fi

\ifnum\backwardsbuckets=1
\subsubsection*{Tolerant testing $C$-approximation}
\newcommand{\twocmone}{m}
\renewcommand{\bucket}{B}
\newcommand{\apx}{C}
\newcommand{\refsum}{s}
We now turn to our second tolerant testing lower bound, which applies to algorithms providing a $C$-factor approximation of the distance to the reference distribution.
\begin{theorem}
	\label{thm:tol-mult-main}
	Any algorithm which, given a reference distribution $\q$ over $[n]$, $\apx \geq 2$, and sample access to an unknown distribution $\p\in \Pi_n(\q)$, distinguishes with probability at least $2/3$ between (i)~$\totalvardist{\p}{\q} \leq \frac{1}{4\apx-1}$ and (ii)~$\totalvardist{\p}{\q} \geq \frac{\apx}{4\apx-1}$, must have sample complexity $\bigOmega{\sqrt{\frac{n}{4^\apx}}}$.
\end{theorem}
\begin{remark}
  As discussed in~\cref{rk:small:tolerance},~\cref{thm:tol-mult-main} is essentially optimal, as it matches (up to polylogarithmic factors in the sample complexity) the upper bound from~\cref{theo:ub:testing} when $C=\Theta(\log\ab)$.
\end{remark}
\begin{proof}
	We will prove the theorem via a sequence of lemmas. 
	\ifnum\withproofs=0
	In the interest of space and exposition, their proofs are deferred to~\cref{sec:proofs-of-misc}. 
	\fi 
	We will assume that $\apx \geq 2$ is an integer, and we define $\twocmone = 2^\apx-1$.  Our proof will proceed similarly to the proof of~\cref{theo:toltesting:lb:1}. We will begin by working over $[\twocmone(2^{\apx+1}+2^{\apx-1}-3)]$.  Throughout this section, we partition $[\twocmone(2^{\apx+1}+2^{\apx-1}-3)]$ into $\apx+1$ buckets, which we will denote $\bucket_0,\bucket_1,\ldots,\bucket_{C}$, such that each $\bucket_i$ is a set of consecutive integers, $|\bucket_\apx| = \twocmone2^{\apx-1}$, $|\bucket_0| = \twocmone$, and $|\bucket_i| = \twocmone2^{i+1}$ for $1 \leq i \leq \apx-1$.  \ignore{We define $\bucket_i = \{b_i,b_i + 1,\ldots,b'_i\}$.}  For convenience, we define $\refsum := \twocmone(4\apx-1)2^{\apx-1}$.
	
	\noindent We define a distribution $\refdist$ in the following way:
	\begin{itemize}
		\item For each $j \in \bucket_0$, $\refdist(j) = \frac{2^{\apx}}{\refsum}$.
		\item For each $1 \leq i \leq \apx-1$ and $j \in \bucket_i$, $\refdist(j) = \frac{2^{C-i}}{\refsum}$.
		\item For each $j \in \bucket_\apx$, $\refdist(j) = \frac{1}{\refsum}$.
	\end{itemize}
	
	We define two distributions $\p$ and $\q$ such that $\p$ and $\q$ are hard to distinguish with few samples, such that $\totalvardist{\refdist}{\p}$ and $\totalvardist{\refdist}{\q}$ are far apart. 
	We define $\q$ in the following way:
	\begin{itemize}
		\item For each $j \in \bucket_0$, $\q(j) = \frac{2^{\apx-1}}{\refsum}$.
		\item For each $1 \leq i \leq \apx-1$,
		\begin{itemize}
			\item For $j$ in the first $\twocmone2^{i}$ elements of $\bucket_i$, \ignore{$\{b_i,\ldots,b_i+\twocmone2^{i}-1\}$,} $\q(j) = \frac{2^{C-i-1}}{\refsum}$.
			\item For $j$ in the next $\twocmone2^{i-1}$ elements of $\bucket_i$, \ignore{$\{b_i+\twocmone2^{i},\ldots,b_i+\twocmone(2^{i} + 2^{i-1}) - 1\}$,} $\q(j) = \frac{2^{C-i}}{\refsum}$.
			\item For $j$ in the last $\twocmone2^{i-1}$ elements of $\bucket_i$, \ignore{$\{b_i+\twocmone(2^{i} + 2^{i-1}),\ldots,b'_i\}$,} $\q(j) = \frac{2^{C-i+1}}{\refsum}$.
		\end{itemize} %
		\item For each $j \in \bucket_\apx$, $\q(j) = \frac{2}{\refsum}$.
	\end{itemize}
	
	\noindent We define $\p$ as follows:
	\begin{itemize}
		\item For each $j \in \bucket_0$,
		\begin{itemize}
			\item \ignore{If $j \in \{b_0,\ldots,b_0+2^{\apx-1}-1\}$,} If $j$ is in the first $2^{\apx-1}$ elements of $\bucket_0$, then $\p(j) = \frac1{\refsum}$.
			\item \ignore{If $j \in \{b_0+2^{\apx-1},\ldots,b'_0\}$,}
			If $j$ is in the last $m-2^{\apx-1} = 2^{\apx-1}-1$ elements of $\bucket_0$, then $\p(j) = \frac{2^\apx}{\refsum}$.
		\end{itemize}
		\item For each $1 \leq i \leq \apx-1$ and $j \in \bucket_i$, $\p(j) = \refdist(j) = \frac{2^{C-i}}{\refsum}$.
		\item For each $j \in \bucket_\apx$,
		\begin{itemize}
			\item \ignore{If $j \in \{b_\apx,\ldots,b_\apx+\twocmone2^{\apx-1}-1\}$,} If $j$ is in the first $(\twocmone-1)2^{\apx-1}$ elements of $\bucket_\apx$, then $\p(j) = \frac1{\refsum}$.
			\item \ignore{If $j \in \{b_\apx+\twocmone2^{\apx-1},\ldots,b'_\apx\}$,} If $j$ is in the last $2^{\apx-1}$ elements of $\bucket_j$, then $\p(j) = \frac{2^\apx}{\refsum}$.
		\end{itemize}
	\end{itemize}
	
	\begin{lemma}
		\label{lem:tol-mult-equal-buckets}
		For $0 \leq i \leq \apx$, $\sum_{j \in \bucket_i} \p(j) = \sum_{j \in \bucket_i} \q(j)$.
	\end{lemma}
	\ifnum\withproofs=1
	\begin{proof}
		The proof is simply direct calculation.	 Observe that in bucket $\apx$, 
		\[
		\refsum \sum_{j \in \bucket_\apx} \q(j) = \twocmone2^{\apx-1} \cdot 2 = (\twocmone-1)2^{\apx-1} + (\twocmone+1)2^{\apx-1} = (\twocmone-1)2^{\apx-1} \cdot 1 + 2^{\apx} \cdot 2^{\apx-1} = \refsum \sum_{j \in \bucket_\apx} \p(j).
		\]		
		In bucket $0$, we have\	\[
		\refsum \sum_{j \in \bucket_0} \q(j) = \twocmone \cdot 2^{\apx-1} = (\twocmone - 1)2^{\apx-1} + 2^{\apx-1} = (2^{\apx} - 2)2^{\apx-1} + 2^{\apx-1} = (2^{\apx-1} - 1) \cdot 2^{\apx} + 2^{\apx-1} \cdot 1 = \refsum \sum_{j \in \bucket_0} \p(j).
		\]
		For $1 \leq i \leq \apx-1$, we have 
		\begin{align*}
			\refsum \sum_{j \in \bucket_i} \p(j) 
			&= \twocmone2^{i+1} \cdot 2^{\apx-i} \\
			&= \twocmone (2^i + 2(2^{i-1}) + 2^{i+1}) 2^{\apx-1-i} \\
			&= \twocmone 2^{i} \cdot 2^{\apx-i-1} + \twocmone2^{i-1} \cdot 2^{\apx-i} + \twocmone2^{i-1} \cdot 2^{\apx-i+1} \\
			&= \refsum \sum_{j \in \bucket_i} \q(j).
		\end{align*}		
		The claim follows by dividing the equalities by $\refsum$.
	\end{proof}
	\fi
	\begin{lemma} 
	\label{lem:tv-distance-far}
		$\totalvardist{\refdist}{\q} = \frac{C}{4C-1}$
	\end{lemma}
	\ifnum\withproofs=1
	\begin{proof}
		By direct calculation,
		\begin{align*}
			2\refsum\totalvardist{\refdist}{\q} &= \refsum \sum_{j=1}^{\refsum} |\refdist(j) - \q(j)| \\
			&= \twocmone 2^{\apx-1}(2-1) + \twocmone(2^\apx - 2^{\apx-1}) + \frac12 \sum_{i=1}^{\apx-1} \left( \twocmone2^{i}(2^{\apx-i} - 2^{\apx-i-1}) + \twocmone2^{i-1}(2^{\apx-i+1}-2^{\apx-i}) \right) \\
			&= \twocmone 2^{\apx} + \sum_{i=1}^{\apx-1} (2^{i-1}\twocmone2^{\apx-i} + \twocmone2^{\apx-1-i}2^i) \\
			&= \twocmone 2^{\apx} + \sum_{i=1}^{\apx-1} (\twocmone2^{\apx-1} + \twocmone2^{\apx-1}) \\
			&= \apx \twocmone 2^{\apx}.
		\end{align*}
		Dividing both sides by $2\refsum$ yields the lemma.
	\end{proof}
	\fi
	\begin{lemma}
		\label{lem:tol-mult-buckets-almost-uniform}
		For every $0 \leq i \leq \apx$, $\p(\bucket_i) \leq \frac{2}{\apx+1}$ (and similarly for $\q(\bucket_i))$.
	\end{lemma}
	\ifnum\withproofs=1
	\begin{proof}
		We apply Lemma~\ref{lem:tol-mult-equal-buckets} and directly calculate.  For bucket $\apx$, we get
		
		\[
		\p(\bucket_\apx) = \q(\bucket_\apx) = \frac{2}{\refsum} \cdot \twocmone{2^\apx-1} = \frac{2}{4\apx-1}.
		\]		
		For bucket $0$, we get
				\[
		\p(\bucket_0) = \q(\bucket_0) = \frac{2^{\apx-1}}{\refsum} \cdot \twocmone = \frac{1}{4\apx-1}.
		\]		
		For $1 \leq i \leq \apx-1$, we get
		\[
		\q(\bucket_i) = \p(\bucket_i) = \frac{2^i}{\refsum} \cdot \twocmone(2^{\apx+1-i}) = \frac{4}{4\apx-1}.
		\]	
	The claim follows by observing that $\frac{4}{4\apx-1} \leq \frac{2}{\apx+1}$ when $\apx \geq \frac32$.  
	\end{proof}
	\fi
	\begin{lemma}
	\label{lem:tv-distance-close} 
		$\totalvardist{\refdist}{\p} = \frac{1}{4\apx-1}$
	\end{lemma}
	\ifnum\withproofs=1
	\begin{proof}
		By direct calculation,
		\[
		2\refsum\totalvardist{\refdist}{\p} = 2^{\apx-1} \cdot (2^\apx - 1) + 2^{\apx-1} \cdot (2^\apx - 1) = 2^{\apx} (2^{\apx-1}) = \twocmone 2^{\apx}.
		\]
		Dividing both sides by $2\refsum$ yields the lemma.
	\end{proof}
	\fi
	\newcommand{\numblk}{t}
	\newcommand{\distsz}{w} 
	Let $\distsz = \twocmone(2^{\apx+1}+2^{\apx-1}-3)$.
	We assume that $n$ is a multiple of $\distsz$, and define $\numblk := \frac{n}{\distsz}$.
	To define $\closefam$ and $\farfam$ over $[n]$, we will need one further piece of notation. We denote by $\mathcal{B}'_{\distsz}\subseteq\mathcal{S}_{\distsz}$ the set of all permutations of $[\distsz]$ ``respecting the buckets,'' that is, for every $0 \leq i \leq \apx$,
	\[
	\mathcal{B}'_{\distsz} = \{ \pi \in \mathcal{S}_{\distsz} : \pi(\bucket_i) = \bucket_i \forall i \in \{0,1,\ldots,C\} \}
	\]
	We then let
	$
	\refdist^* := \frac{1}{\numblk} (\refdist \sqcup \refdist \sqcup \cdots \sqcup \refdist)
	$
	as well as
	\begin{align*}
		\closefam &= \setOfSuchThat{  \frac{1}{\numblk}(\closedist\circ\pi_1 \sqcup \closedist\circ\pi_2 \sqcup \dots \sqcup \closedist\circ\pi_\numblk) }{ \pi_1,\dots, \pi_\numblk \in \mathcal{B'}_{\distsz} }
		\\
		\farfam &= \setOfSuchThat{  \frac{1}{\numblk}(\fardist\circ\pi_1 \sqcup \fardist\circ\pi_2 \sqcup \dots \sqcup \fardist\circ\pi_\numblk) }{ \pi_1,\dots, \pi_\numblk \in \mathcal{B'}_{\distsz} }
	\end{align*}
	where as before $\sqcup$ denotes vector concatenation.  Since $\totalvardist{\refdist}{\closedist \circ \pi} = \totalvardist{\refdist}{\closedist}$ and $\totalvardist{\refdist}{\fardist \circ \pi} = \totalvardist{\refdist}{\fardist}$ for all $\pi \in \mathcal{B}'_{\refsum}$, we have that $\totalvardist{\refdist^*}{\p} = \frac{1}{4\apx-1}$ for every distribution $\p \in \closefam$, and $\totalvardist{\refdist^*}{\q} = \frac{\apx}{4\apx-1}$ for every distribution $\q \in \farfam$.  Further, repeating the same partitioning of each interval of $\refsum$ elements of $[n]$ into buckets $\bucket_0,\bucket_1,\ldots,\bucket_C$, we have $\numblk (\apx+1)$ buckets, such that distinguishing a distribution in $\closefam$ from a distribution in $\farfam$ requires seeing at $2$ samples in at least one of these buckets.
	Since the probability mass on each of the buckets at most $\frac{2}{\numblk(\apx+1)}$ by Lemma~\ref{lem:tol-mult-buckets-almost-uniform}, at least $\Omega(\sqrt{\numblk(\apx+1)}) = \Omega(\sqrt{n(\apx+1)/\distsz})$ queries to distinguish in $\closefam$ from a distribution in $\farfam$, completing the proof of Theorem~\ref{thm:tol-mult-main}.
\end{proof}
\fi
\printbibliography
\ifnum\withproofs=0
\appendix
\section{Omitted proofs}
	\label{sec:proofs-of-misc}
\input{sec-proofs-of-misc-lemmas}
\fi
\end{document}